\documentclass[11pt]{amsart}

\author{Carlo Sanna}
\address{Universit\`a degli Studi di Torino\\Department of Mathematics\\Turin, Italy}
\email{carlo.sanna.dev@gmail.com}

\title{On the sum of digits of the factorial}

\usepackage{amsmath}
\usepackage{amssymb}
\usepackage{amsthm}
\usepackage{geometry}
\usepackage{color}
\usepackage{hyperref}
\usepackage{enumerate}
\hypersetup{colorlinks=true}

\newtheorem{thm}{Theorem}[section]

\newtheorem{lem}[thm]{Lemma}
\theoremstyle{definition}

\begin{document}

\begin{abstract}
Let $b \geq 2$ be an integer and denote by $s_b(m)$ the sum of the digits of the positive integer $m$ when is written in base $b$.
We prove that $s_b(n!) > C_b \log n \log \log \log n$ for each integer $n > e$, where $C_b$ is a positive constant depending only on $b$.
This improves of a factor $\log \log \log n$ a previous lower bound for $s_b(n!)$ given by Luca.
We prove also the same inequality but with $n!$ replaced by the least common multiple of $1,2,\ldots,n$.
\end{abstract}

\maketitle

\section{Introduction}

Let $b \geq 2$ be an integer and denote by $s_b(m)$ the sum of the digits of the positive integer $m$ when is written in base $b$.
Lower bounds for $s_b(m)$ when $m$ runs through the member of some special sequence of natural numbers (e.g., linear recurrence sequences \cite{Ste80} \cite{Luc00} and sequences with combinatorial meaning \cite{LS10} \cite{LS11} \cite{KL12} \cite{Luc12}) have been studied before.

In particular, Luca \cite{Luc02} showed that the inequality
\begin{equation}\label{equ:luca_bound}
s_b(n!) > c_b \log n ,
\end{equation}
holds for all the positive integers $n$, where $c_b$ is a positive constant, depending only on $b$.
He also remarked that (\ref{equ:luca_bound}) remains true if one replaces $n!$ by 
\begin{equation*}
\Lambda_n := \operatorname{lcm}(1,2,\ldots,n),
\end{equation*}
the least common multiple of $1,2,\ldots,n$.
We recall that $\Lambda_n$ has an important role in elementary proofs of the Chebyshev bounds $\pi(x) \asymp x / \log x$, for the prime counting function $\pi(x)$ \cite{Nai82}.

In this paper, we give a slight improvement of (\ref{equ:luca_bound}) by proving the following
\begin{thm}\label{thm:sb_bound}
For each integer $n > e$, it results
\begin{equation*}
s_b(n!), \, s_b(\Lambda_n) > C_b \log n \log \log \log n ,
\end{equation*}
where $C_b$ is a positive constant, depending only on $b$.
\end{thm}

\section{Preliminaries}

In this section, we discuss a few preliminary results needed in our proof of Theorem~\ref{thm:sb_bound}.
Let~$\varphi$ be the Euler's totient function.
We prove an asymptotic formula for the maximum of the preimage of $[1,x]$ through $\varphi$, as $x \to +\infty$.
Although the cardinality of the set $\varphi^{-1}([1,x])$ is well studied \cite{Bat72} \cite{BS90} \cite{BT98}, in the literature we have found no results about $\max(\varphi^{-1}([1,x]))$ as our next lemma.

\begin{lem}\label{lem:max_m_phi}
For each $x \geq 1$, let $m = m(x)$ be the greatest positive integer such that $\varphi(m) \leq x$.
Then $m \sim e^\gamma x \log \log x$, as $x \to +\infty$, where $\gamma$ is the Euler--Mascheroni constant.
\end{lem}
\begin{proof}
Since $\varphi(n) \leq n$ for each positive integer $n$, it results $m \geq \lfloor x \rfloor$.
In particular, $m \to +\infty$ as $x \to +\infty$. 
Therefore, since the minimal order of $\varphi(n)$ is $e^{-\gamma} n /\log \log n$ (see \cite[Chapter~I.5, Theorem~4]{Ten95}), we obtain
\begin{equation*}
(e^{-\gamma} + o(1))\frac{m}{\log \log m} \leq \varphi(m) \leq x,
\end{equation*}
as $x \to +\infty$.
Now $\varphi(n) \geq \sqrt{n}$ for each integer $n \geq 7$, thus $m \leq x^2$ for $x \geq 7$.
Hence,
\begin{equation*}
m \leq (e^\gamma + o(1)) \, x \log \log m \leq (e^\gamma + o(1)) \, x \log \log (x^2) = (e^\gamma + o(1)) \, x \log \log x ,
\end{equation*}
as $x \to +\infty$.

On the other hand, let $p_1 < p_2 < \cdots$ be the sequence of all the (natural) prime numbers and let $a_1 < a_2 < \cdots$ be the sequence of all the $3$-smooth numbers, i.e., the natural numbers of the form $2^a 3^b$, for some integers $a,b \geq 0$.
Moreover, let $s = s(x)$ be the greatest positive integer such that
\begin{equation*}
(p_1 - 1)\cdots(p_s - 1) \leq \sqrt{x} ,
\end{equation*}
and let $t = t(x)$ be the greatest positive integer such that
\begin{equation*}
a_t (p_1 - 1)\cdots(p_s - 1) \leq x .
\end{equation*}
Note that $s, t \to +\infty$ as $x \to +\infty$.
Now we have (see \cite[Chapter~I.1, Theorem~4]{Ten95})
\begin{equation*}
\sqrt{x} < (p_1 - 1)\cdots(p_{s+1} - 1) < p_1 \cdots p_{s+1} \leq 4^{p_{s+1}} ,
\end{equation*}
hence
\begin{equation}\label{equ:ps_bound}
p_s > \tfrac1{2}p_{s+1} > \tfrac1{4\log 4} \log x ,
\end{equation}
from Bertrand's postulate.
Put $m^\prime := a_t p_1 \cdots p_s$, so that for $s \geq 2$ we get
\begin{equation*}
\varphi(m^\prime) = a_t (p_1 - 1) \cdots (p_s - 1) \leq x,
\end{equation*}
and thus $m \geq m^\prime$.
By a result of P\'olya \cite{Pol18}, $a_t / a_{t+1} \to 1$ as $t \to +\infty$.
Therefore, from our hypothesis on $s$ and $t$, Mertens' formula \cite[Chapter~I.1, Theorem~11]{Ten95} and (\ref{equ:ps_bound}) it follows that
\begin{align*}
m \geq m^\prime &= \frac{a_t}{a_{t+1}} \cdot a_{t+1} \prod_{i=1}^s (p_i-1) \cdot \prod_{i=1}^s \left(1 - \frac1{p_i}\right)^{-1} > (1 + o(1)) \cdot x \cdot \frac{\log p_s}{e^{-\gamma} + o(1)} \\
&> (e^\gamma + o(1)) \,x \log \log x ,
\end{align*}
as $x \to +\infty$.
\end{proof}

Actually, we do not make use of Lemma~\ref{lem:max_m_phi}.
We need more control on the factorization of a ``large'' positive integer $m$ such that $\varphi(m) \leq x$, 
even at the cost of having only a lower bound for $m$ and not an asymptotic formula.

\begin{lem}\label{lem:max_m_Q}
For each $x \geq 1$ there exists a positive integer $m = m(x)$ such that: $\varphi(m) \leq x$; $m = 2^t Q$, where $t$ is a nonnegative integer and $Q$ is an odd squarefree number; and
\begin{equation*}
m \geq (\tfrac1{2} e^\gamma + o(1))\, x \log \log x , 
\end{equation*}
as $x \to +\infty$.
\end{lem}
\begin{proof}
The proof proceeds as the second part of the proof of Lemma~\ref{lem:max_m_phi}, but with $a_k := 2^{k-1}$ for each positive integer $k$.
So instead of $a_t / a_{t + 1} \to 1$, as $t \to +\infty$, we have $a_t / a_{t + 1} = 1/2$ for each $t$.
We leave the remaining details to the reader.
\end{proof}

To study $\Lambda_n$ is useful to consider the positive integers as a poset ordered by the divisibility relation $\mid$.
Thus, obviously, $\Lambda_n$ is a monotone nondecreasing function, i.e., $\Lambda_m \mid \Lambda_n$ for each positive integers $m \leq n$.
The next lemma says that $\Lambda_n$ is also super-multiplicative.

\begin{lem}\label{lem:L_div}
We have $\Lambda_m \Lambda_n \mid \Lambda_{mn}$, for any positive integers $m$ and $n$.
\end{lem}
\begin{proof}
It is an easy exercise to prove that
\begin{equation*}
\Lambda_n = \prod_{p \,\leq\, n} p^{\lfloor \log_p n\rfloor} ,
\end{equation*}
for each positive integer $n$, where $p$ runs over all the prime numbers not exceeding $n$.
Therefore, the claim follows since
\begin{equation*}
\lfloor \log_p m \rfloor + \lfloor \log_p n\rfloor \leq \lfloor \log_p m + \log_p n\rfloor = \lfloor \log_p mn\rfloor ,
\end{equation*}
for each prime number $p$.
\end{proof}

We recall some basic facts about cyclotomic polynomials.
For each positive integer $n$, the $n$-th cyclotomic polynomial $\Phi_n(x)$ is defined by
\begin{equation*}
\Phi_n(x) := \prod_{\substack{1 \,\leq\, k\, \leq\, n \\ \gcd(k,n) = 1}} \left(x - e^{2\pi i k / n}\right) .
\end{equation*}
It results that $\Phi_n(x)$ is a polynomial with integer coefficients and that it is irreducible over the rationals, with degree $\varphi(n)$.
We have the polynomial identity
\begin{equation*}
x^n - 1 = \prod_{d \,\mid\, n} \Phi_d(x) ,
\end{equation*}
where $d$ runs over all the positive divisors of $n$.
Moreover, it holds $\Phi_n(a) \leq (a + 1)^{\varphi(n)}$, for all $a \geq 0$.
The next lemma regards when $\Phi_m(a)$ and $\Phi_n(a)$ are not coprime.

\begin{lem}\label{lem:Phi_coprime}
Suppose that $\gcd(\Phi_m(a), \Phi_n(a)) > 1$ for some integers $m,n,a \geq 1$.
Then $m / n$ is a prime power, i.e., $m / n = p^k$ for a prime number $p$ and an integer $k$.
\end{lem}
\begin{proof}
See \cite[Theorem~7]{Ge08}.
\end{proof}

Finally, we state an useful lower bound for the sum of digits of the multiples of $b^m - 1$.
\begin{lem}\label{lem:sbbm1}
For each positive integers $m$ and $q$ it results $s_b((b^m - 1)q) \geq m$.
\end{lem}
\begin{proof}
See \cite[Lemma~1]{BD12}.
\end{proof}

\section{Proof of Theorem~\ref{thm:sb_bound}}

Without loss of generality, we can assume $n$ sufficiently large.
Put $x := \tfrac1{8}\log_{b+1} n \geq 1$.
Thanks to Lemma \ref{lem:max_m_Q}, we know that there exists a positive integer $m$ such that $\varphi(m) \leq x$ and
\begin{equation}\label{equ:m_bound}
m > \tfrac1{3} e^\gamma x \log \log x > C_b \log n \log \log \log n ,
\end{equation}
where $C_b > 0$ is a constant depending only on $b$.
Precisely, we can assume that $m = 2^t Q$, where $t$ is a nonnegative integer and $Q$ is an odd squarefree number.
Fix a nonnegative integer $j \leq t$.
For each positive divisor $d$ of $Q$, we have $\varphi(2^{t - j}d) \mid \varphi(m / 2^j)$ and so, a fortiori, $\varphi(2^{t - j}d) \leq \varphi(m / 2^j)$.
Therefore,
\begin{equation}\label{equ:Phi_d_bound}
\Phi_{2^{t-j}d}(b) \leq (b + 1)^{\varphi(2^{t-j}d)} \leq (b + 1)^{\varphi(m / 2^j)} \leq (b + 1)^{\varphi(m) / 2^{j-1}} \leq n^{1 / 2^{j+2}} .
\end{equation}
Let $\mu$ be the M\"obius function.
Now from (\ref{equ:Phi_d_bound}) and Lemma~\ref{lem:Phi_coprime} we have that the $\Phi_{2^{t-j}d}(b)$'s, where $d$ runs over the positive divisors of $Q$ such that $\mu(d) = 1$,  are pairwise coprime and not exceeding $n^{1 / 2^{j+2}}$, thus
\begin{equation}\label{equ:Phi_div}
\prod_{\substack{d \, \mid \, Q \\ \mu(d) \,=\, 1}} \Phi_{2^{t-j}d}(b) = \operatorname{lcm}\{\Phi_{2^{t-j}d}(b) : d \mid Q, \; \mu(d) = 1\} \mid \Lambda_{\lfloor n^{1 / 2^{j+2}} \rfloor} .
\end{equation}
Similarly, the same result holds for the divisors $d$ such that $\mu(d) = -1$.
Clearly, we have
\begin{equation*}
b^m - 1 = \prod_{d \,\mid\, m} \Phi_{d}(b) = \prod_{\substack{0 \,\leq\, j \,\leq\, t \\ r \, \in \, \{-1,+1\}}} \prod_{\substack{d \, \mid \, Q \\ \mu(d) \,=\, r}} \Phi_{2^{t-j}d}(b) .
\end{equation*}
Moreover,
\begin{equation*}
\left(\prod_{0 \,\leq\, j \,\leq\, t} \lfloor n^{1/2^{j+2}}\rfloor \right)^{\!\!2} \leq \prod_{0 \,\leq\, j \,\leq\, t} n^{1/2^{j+1}} \leq n .
\end{equation*}
As a consequence, from (\ref{equ:Phi_div}) and Lemma~\ref{lem:L_div}, we obtain
\begin{equation*}
b^m - 1 \mid \left(\prod_{0 \,\leq\, j \,\leq\, t} \Lambda_{\lfloor n^{1/2^{j+2}}\rfloor}\right)^{\!\!2} \mid \Lambda_n .
\end{equation*}
Thus $b^m - 1 \mid \Lambda_n$ and also $b^m - 1 \mid n!$, since obviously $\Lambda_n \mid n!$.
In conclusion, from Lemma~\ref{lem:sbbm1} and (\ref{equ:m_bound}), we get
\begin{equation*}
s_b(\Lambda_n),\, s_b(n!) \geq m > C_b \log n \log \log \log n,
\end{equation*}
which is our claim, this completes the proof.

\subsection*{Acknowledgements}

The author thanks Paul Pollack (University of Georgia) for a suggestion that has lead to the exact asymptotic formula of Lemma~\ref{lem:max_m_phi}.

\bibliographystyle{amsalpha}
\providecommand{\bysame}{\leavevmode\hbox to3em{\hrulefill}\thinspace}
\providecommand{\MR}{\relax\ifhmode\unskip\space\fi MR }

\providecommand{\href}[2]{#2}

\end{document}